\documentclass[11pt,english,reqno]{amsart}

\usepackage[top=3cm,bottom=3.5cm,inner=3cm,outer=3cm]{geometry}
\usepackage{setspace}
\usepackage[pagebackref]{hyperref}
\usepackage{enumerate}
\usepackage{amsmath,amsfonts,amssymb,amsthm}
\usepackage[T1]{fontenc}
\usepackage{caption}
\usepackage{graphicx}
\usepackage{csquotes}

\usepackage{tikz} 
\usetikzlibrary{decorations.pathreplacing, calc, arrows, decorations.markings}

\MakeOuterQuote{"}

\newtheorem{theorem}{Theorem}

\newtheorem{claim}[theorem]{Claim}

\newtheorem{lemma}[theorem]{Lemma}

\newcommand{\F}{\mathcal{F}}
\newcommand{\G}{\mathcal{G}}
\newcommand{\T}{\mathcal{T}}

\begin{document}

\title{$\F$-Saturation Games}

\author[J. D. Lee]{Jonathan D. Lee}
\address{Department of Pure Mathematics and Mathematical Statistics, University of Cambridge, Wilberforce Road, Cambridge CB3\thinspace0WB, UK}
\email{j.d.lee@dpmms.cam.ac.uk}

\author[A. Riet]{Ago-Erik Riet}
\address{Institute of Mathematics, University of Tartu, Juhan Liivi 2, 50409 Tartu, Estonia}
\email{ago-erik.riet@ut.ee}
\thanks{The research was partially supported by institutional research funding IUT20-57 of the Estonian Ministry of Education and Research.}

\date{\today}

\begin{abstract}
We study \emph{$\F$-saturation games}, first introduced by F\"uredi, Reimer and Seress~\cite{FuReSe} in 1991, and named as such by West~\cite{We}. The main question is to determine the length of the game whilst avoiding various classes of graph, playing on a large complete graph. We show lower bounds on the length of path-avoiding games, and more precise results for short paths. We show sharp results for the tree avoiding game and the star avoiding game.
\end{abstract}

\maketitle

\section{Introduction}

For $\F$ a family of graphs, we say a graph $G$ is \emph{$\F$-free} if $G$ contains no member of $\F$ as a subgraph. We say $G\subset H$ is a \emph{$\F$-saturated} subgraph of $H$ if $G$ is a maximal $\F$-free subgraph of $H$. For a discussion of saturated graphs see for example Bollob\'as~\cite{Bo}. Take $H$ a graph, $|H| = n$, and let $\F$ be a family of graphs. Following the definition of the triangle free game of F\"uredi, Reimer and Seress~\cite{FuReSe}, and building on the notation of West\cite{We}, we define the \emph{$\F$-saturation game} as follows.

We have two players, \emph{Prolonger} and \emph{Shortener}, who we take to be male and female respectively. We define a graph process $\G_i$. We initially set $\G_0 = E_n$, the empty graph on $n$ vertices. The process ends at time $t^*$ if $G_{t^*}$ is a $\F$-saturated subgraph of $H$. Otherwise, at time $2t$ , Prolonger chooses an edge $uv \in H \backslash \G_{2t}$ and $\G_{2t}\cup uv$ is $\F$-free, and $\G_{2t+1} = G_{2t} \cup uv$. Similarly, at time $2t+1$ Shortener chooses an edge from $H\backslash \G_{2t+1}$ to add, such that the process remains $\F$-free. Prolonger's goal is to maximise $t^*$, whilst Shortener wishes to minimise $t^*$. Our results will not depend on which of the two players moves first, and so we refer to this game as $\G(H;\F)$. We say the value of $t^*$ under optimal play by both Prolonger and Shortener is the \emph{score} or \emph{game saturation number} of $\G(H;\F)$, denoted by $\textbf{Sat}(\G)$ or simply $\G$ provided there can be no confusion. When only one graph is excluded, we write $\G(H;F) := \G(H;\{F\})$.

F\"uredi, Reimer and Seress~\cite{FuReSe} concentrate on the game $\G(K_n,K_3)$. They exhibit
a strategy for Prolonger which demonstrates that $\G(K_n,K_3) \geq (\frac{1}{2} + o(1))n \log_2 n$. They attribute to Erd\H os a lost proof that Shortener has a strategy showing $\G(K_n,K_3) \leq \frac{n^2}{5}$. Bir\'o, Horn and Wildstrom~\cite{BiHoWi} show that $\G(K_n,K_3) \leq \frac{9n^2}{50}$, by showing that Shortener is able to cover almost all vertices with disjoint cycles of length 5, and that the number of edges between two such cycles is bounded by 18.

Motivated by these results, we study the case where $\F = \{P_k\}$, $\F$ is the class of all trees on $k$ vertices or $\F = \{K_{1,k}\}$.

\section{Our results}

As is standard, for any $k \in \mathbb{N}$ we denote a path on $k$ vertices by $P_k$. To illustrate the difficulties encountered by Prolonger, we first study a variant where on his turn, he is permitted to decline to pick any edge, and set $\G_{2t+1} = \G_{2t}$. Since Shortener is still required to add edges, this process will still become $\F$-saturated and thus have a score as defined for the $\F$-saturation game. We will refer to this game as $\G_{-P}$. Since we have given Prolonger additional options, it is clear that any strategy he might use in $\G$ is valid in $\G_{-P}$, and so we have that $\G_{-P}(H;\F) \geq \G(H;\F)$

\begin{theorem}\label{G_P-}
For all $n \geq k$, we have $\frac{1}{4}n(k - 2) \leq \G_{-P}(K_n; P_k) \leq \frac{1}{2}n(k-1)$.
\end{theorem}

Returning to $\G(K_n, P_k)$, we have results only for small values of $k$. Whilst these results are quite precise, they are predicated on a complete categorisation of the connected $P_k$-saturated graphs. Obtaining results of this precision for larger $k$ thus seems challenging. Recently Carraher, Kinnersley, Reiniger and West \cite{Ca} made us aware of another proof, only in the $P_4$ case, which gives a slight improvement on the additive constants.
\begin{theorem}\label{P4-bound}
For all $n > 0$, we have $\frac{4}{5}n-\frac{8}{5} \leq \G(K_n, P_4) \leq \frac{4}{5}n + 1$.
\end{theorem}
\begin{theorem}\label{P5-bound}
For all $n > 0$, we have $n-1 \leq \G(K_n,P_5) \leq n+2$.
\end{theorem}

For larger classes of graphs, we have substantially precise bounds for all $k$. We define $\T_k$ to be the family of all trees on $k$ vertices. 
\begin{theorem}\label{Tree-bound}
For all $n$, $k$, we have:
\[\begin{array}{lll} \\
\G(K_n, \mathcal{T}_k) &
  \begin{array}{l}
  = \lfloor\frac{n}{k-1}\rfloor {\binom{k-1}{2}} + {\binom{n-(k-1)\lfloor\frac{n}{k-1}\rfloor}{2}}
  \end{array}
& n \not\equiv 1 \mod (k-1)\\
\G(K_n, \mathcal{T}_k) &
  \begin{array}{l}
  \leq \frac{n}{k-1} {\binom{k-1}{2}} \\
  \geq \frac{n}{k-1} {\binom{k-1}{2}} - (k-3)
  \end{array}
& n \equiv 1 \mod (k-1),\; k \geq 3.\\
\end{array}\]
\end{theorem}

Carraher, Kinnersley, Reiniger and West \cite{Ca} consider the case $n=k$ of the preceding Theorem and determine the precise score. In fact, the primary constraint of $\T_k$ saturation is to exclude a $K_{1,k-1}$. If only this graph is excluded, we have a precise bound:
\begin{theorem}\label{Star-bound}
For $n \geq (3k+1)(k-2)$, we have the following bounds:
\[
\frac{1}{2}kn \geq \G(K_n, K_{1,k+1}) \geq \frac{1}{2}\left(kn-2(k-1)\right). 
\]
\end{theorem}

\section{Avoiding \texorpdfstring{$P_k$}{P\_k} in \texorpdfstring{$\G_{-P}$}{\G\_-P}}

\begin{proof}[Proof of Theorem \ref{G_P-}]
The upper bound is the saturation result of Erd\H{o}s and Gallai from \cite{ErGa}. To obtain the lower bound, we exhibit a strategy for Prolonger that guarantees the required length of game. We say a graph is \emph{everywhere traceable} if for every vertex $v$ in the graph there is a Hamiltonian path starting at $v$. Hence if a graph is Hamiltonian it is everywhere traceable. We will show that the following strategy for Prolonger guarantees that the score will be large enough:
\begin{enumerate}[i)]
\item If there is a component $C$ which is not everywhere traceable, he finds a Hamiltonian path $P$ in $C$ and adds the edge which augments $P$ to a Hamiltonian cycle;
\item Otherwise he does not add an edge on his turn.
\end{enumerate}

To prove this, we first show the following auxiliary claim.
\begin{claim}
After his move, every connected component is everywhere traceable.
\end{claim}
\begin{proof}
Induct on the number of edges in the graph. Hence we may assume that after his previous move, all the components were everywhere traceable. As the base case, note that the empty graph and an isolated edge are everywhere traceable, so regardless of who moves first, Prolonger will choose to add no edges and leave the graph satisfying the claim.

After Shortener's move, Prolonger is faced with a graph $G$. If her move did not alter the component structure of the process, then every component is still everywhere traceable and he will add no edges, satisfying the claim. Her move altered at most 2 components by connecting them, which produces a single component $C$ which is not everywhere traceable. Since $C$ was formed by joining two everywhere traceable components by an edge, we know that $C$ contains a Hamiltonian path $P$. Since after her move the graph is $P_k$-free, we know that $|P| = |C| < k$.

Since by assumption $C$ is not everywhere traceable, we have that $|C| > 2$ and that the endpoints $u, v$ of $P$ are not adjacent as $C$ is not Hamiltonian. Since $|C| < k$, any path using the edge $uv$ in $G \cup uv$ is contained in $V(C)$, and so has length less than $k$. So $G \cup uv$ is $P_k$-free, and Prolonger may add this edge. The component $C \cup uv$ is Hamiltonian and thus everywhere traceable. Hence after his move, Prolonger leaves every connected component everywhere traceable.
\end{proof}

Hence in $\G_{t^*}$ , the total number of vertices in any two components is $\geq k$, as otherwise these two components could be joined by an edge. Since all components are Hamiltonian, every component is of size less than $k$ and so will be complete. Hence the sum of degrees of any two disconnected vertices is at least $k - 2$. Hence taking $\delta$ the minimum degree of $\G_{t^*}$, we have:
\[2\G \geq \max(k - 2 - \delta, \delta)(n - \delta - 1) + \delta(\delta + 1) = \max(k - 2 - 2\delta, 0)(n - \delta - 1) + \delta n\]
which is minimised by taking $\delta = \lfloor\frac{k - 2}{2}\rfloor$. Checking $k$ even or odd, and recalling that $k < n$ yields $\G \geq\frac{1}{4}n(k-2)$ as required.
\end{proof}

In fact, the notion of ensuring that all components remain everywhere traceable almost allows for an effective strategy for Prolonger in $\G(K_n; P_k)$. The only point at which Prolonger could not guarantee to leave every component everywhere traceable is when when the graph consists of a disjoint union of cliques. In this case, his move necessarily leaves a component which is not everywhere traceable. This could be exploited by Shortener to produce a large induced star with each vertex attached to a long path, and any path between vertices of the star passing through the central vertex. This permits disconnected vertices to have degrees summing to less than $k-1$, and thus in principle to push $\G$ below $\frac{1}{4}n(k-2)$.

\section{The game \texorpdfstring{$\G(K_n, P_4)$}{G(K\_n, P\_4)}}

We now turn to a detailed examination of the game $\G(K_n, P_4)$ and Theorem~\ref{P4-bound}. Let us begin with the following characterisation of $P_4$-saturated graphs, which is easily seen by inspection:

\begin{lemma}\label{P4_saturated}
A $P_4$ saturated graph is either a vertex-disjoint union of triangles and stars with at least two vertices or a vertex-disjoint union of triangles and an isolated vertex (cf. Figure \ref{P4-cmpts}).
\end{lemma}
\begin{proof}
\end{proof}

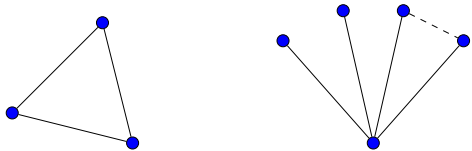
\begin{figure}[h]
\scalebox{.8}{
  \begin{tikzpicture}
    \tikzstyle{every node}=[draw,circle,fill=blue,minimum size=0.2cm,
                            inner sep=0pt]
    \node (p1) at ( 0, 0.5) {}; 
    \node (p2) at ( 2, 0) {};
    \node (p3) at ( 1.5,2) {};
    \node (q1) at ( 6,0) {};
    \node (q2) at ( 4.5,1.7) {};
    \node (q3) at ( 5.5, 2.2) {};
    \node (q4) at ( 6.5,2.2) {};
    \node (q5) at ( 7.5,1.7) {};
    \begin{scope}[every path/.style={-}]
       \draw (p1) -- (p2);
       \draw (p2) -- (p3);
       \draw (p3) -- (p1);
       \draw (q1) -- (q2);
       \draw (q1) -- (q3);
       \draw (q1) -- (q4);
       \draw (q1) -- (q5);
       \draw[dashed] (q4) -- (q5);

    \end{scope}  
  \end{tikzpicture}
}
\captionof{figure}{Maximal components of a $P_4$ saturated graph}
\label{P4-cmpts}
\end{figure}

This straightforward lemma leads to reasonably good bounds on the score, as we can exactly track which components could form.

\begin{proof}[Proof of Theorem \ref{P4-bound}]
The upper bound is demonstrated by considering the following strategy for Shortener. She will:

\begin{enumerate}[i)]
\item extend a $K_{1,2}$ to a $K_{1,3}$ if possible, otherwise 
\item draw an isolated edge if possible, otherwise
\item extend a star by attaching the central vertex to an isolated vertex if possible, otherwise
\item extend a $K_{1,2}$ to a $K_3$.
\end{enumerate}

\begin{claim}\label{P4-structure}
After Prolonger's move, there is at most one $K_{1,2}$ component.
Shortener will not complete the $K_{1,2}$ to a $K_3$, unless this makes the graph
$P_4$ saturated. After Shortener's move, there is at most one $K_{1,2}$ component.
If there is a $K_{1,2}$ component after Shortener's move, Prolonger will extend it to a $K_3$ and make the
graph $P_4$ saturated.\end{claim}
\begin{proof}
We proceed by induction.
\begin{enumerate}[1.]
\item Suppose that $G_i$ has a $K_{1,2}$ component after Prolonger's move.. 
\begin{enumerate}[1)] 
\item If there is an isolated vertex in $G_i$, Shortener will extend the $K_{1,2}$ to a
$K_{1,3}$. Hence there is no $K_{1,2}$ component in $G_{i+1}$, and there can be at
most one after Prolonger's next move to $G_{i+2}$.
\item If there is no isolated vertex in $G_i$, Shortener will extend the $K_{1,2}$ to a
$K_3$. Two components of size $> 1$ cannot be joined without creating a
$P_4$. Hence no further components can be joined or extended, and the
graph is $P_4$ saturated.
\end{enumerate}
\item Suppose that $G_i$ has no $K_{1,2}$ component after Prolonger's move.
\begin{enumerate}[1)]
\item If Shortener creates a $K_{1,2}$ component, then $G_i$ contained exactly one
isolated vertex. Hence all components are now of size > 1, so
Prolonger can only complete the $K_{1,2}$ to a $K_3$.
\item Otherwise there are no $K_{1,2}$ components in $G_{i+1}$, and Prolonger can
produce at most one in $G_{i+2}$.
\end{enumerate}
\end{enumerate}
This finishes the proof.
\end{proof}
By Claim \ref{P4-structure}, until $t^*$, Shortener has ensured that the graph is a vertex disjoint union of stars. Let there be $\lambda$ components in $G_{t^*}$. Since there is at most 1 triangle, the score is bounded above by $n + 1 - \lambda$, with $n - \lambda$ moves producing non-trivial components (i.e. creating isolated edges) or extending stars. To prevent her from making a new non-trivial component by case (ii) of her strategy, Prolonger must make a $K_{1,2}$, which occurs at most once for each component of $\G_{t^*}$. Hence at most $\lambda$ of Shortener's moves fail to make a non-trivial component. Hence there are at least $\frac{1}{2} (n - \lambda) - \lambda$ components. So $\lambda\geq \frac{1}{5}n$, and the score is at most $\frac{4}{5}n + 1$.

The lower bound is demonstrated by considering the following strategy
for Prolonger. He will:
\begin{enumerate}[i)]
\item complete a triangle component if possible, otherwise 
\item complete a $K_{1,2}$ component if possible, otherwise
\item extend a star component if possible, otherwise 
\item draw an isolated edge. 
\end{enumerate}
Note that Prolonger is forced to play an isolated edge only as the first move or after Shortener completes a triangle.  We say that a move \emph{uses $k$ new vertices} if the number of isolated vertices is reduced by $k$ as a result of that move.

We first claim that if Prolonger creates a $K_{1,2}$ component in $\G_i$, at most 2 isolated vertices are used between $\G_i$ and $\G_{i+2}$. If Shortener plays elsewhere, Prolonger will extend the $K_{1,2}$ to a $K_3$. If Shortener extends the $K_{1,2}$ to a $K_3$, Prolonger can make an arbitrary move. If Shortener extends the $K_{1,2}$ to a $K_{1,3}$ then Prolonger can extend that to a $K_{1,4}$. In all cases at most two new vertices are used.

Note that if Prolonger can create a $K_{1,2}$ component when creating $\G_i$ but does not, then he must extend a $K_{1,2}$ into a $K_3$. Hence at most 2 new vertices are used between $\G_{i-2}$ and $\G_i$. 

If Prolonger cannot create a $K_3$ or $K_{1,2}$ component then either there are no isolated edges in $\G_{i-1}$, or there are no isolated vertices in $\G_{i-1}$. Hence Shortener uses at most one isolated vertex from $\G_{i-2}$ or $\G_{i-1}$ is $P_4$ saturated. Prolonger uses 2 new vertices only if he adds an isolated edge to form $\G_i$, which requires that Shortener completed a triangle into $\G_{i-1}$ and used no new vertices. Hence either 1 new vertex is used to end the game or at most 2 new vertices are used between $\G_{i-2}$ and $\G_i$. 

Note that with this strategy of Prolonger when $\G_i$ is created from $\G_{i-2}$ we never use 4 new vertices. Furthermore, we use 3 new vertices only if in $\G_{i-2}$ there was no $K_{1,2}$ component and in $\G_i$ there is. As a consequence, no two consecutive pairs of moves by Shortener and then Prolonger both use 3 new vertices. If Prolonger moves first, then his first move consumes two new vertices. If Shortener makes the last move last then her move may consume two new vertices. If there are an odd number of pairs of moves by Shortener and then Prolonger we may have 1 more pair using 3 new vertices than 2. In the worst case, all of these occur, the score would be $2(2k+1)+2$ with $n = 5k + 3 + 2 + 2 = 5k + 7$, and so the score would be $\frac{4}{5}(n-2)$. This completes the proof of the lower bound.
\end{proof}

\section{The game \texorpdfstring{$\G(K_n, P_5)$}{G(K\_n, P\_5)}}

We now turn to a detailed examination of the game $\G(K_n, P_5)$ and Theorem~\ref{P5-bound}. Denote a double star with $k$ pendant edges at one end of the central edge and $l$ at the other by $D_{k,l}$. Denote a triangle with $k$ pendant edges at one vertex by $T_k$ (cf. Figure~\ref{dt}).
\begin{figure}[h]
\scalebox{.8}{
  \begin{tikzpicture}
    \tikzstyle{vertex}=[draw,circle,fill=blue,minimum size=0.2cm,inner sep=0pt]

    \node[vertex] (r1) at ( 11,0) {};
    \node[vertex] (r2) at ( 9.3,-1.5) {};
    \node[vertex] (r3) at ( 8.8, -0.5) {};
    \node[vertex] (r4) at ( 8.8,0.5) {};
    \node[vertex] (r5) at ( 9.3,1.5) {};
    \node[vertex] (r6) at ( 13, 1.2) {};
    \node[vertex] (r7) at ( 13, -1.2) {};

    \begin{scope}[every path/.style={-}]
       \draw (r1) -- (r2);
       \draw (r1) -- (r3);
       \draw (r1) -- (r4);
       \draw (r1) -- (r5);
       \draw[dashed] (r4) -- (r5);
       \draw (r6) -- (r7);
       \draw (r7) -- (r1);
       \draw (r1) -- (r6);
     \end{scope} 
\draw [decorate,decoration={brace,amplitude=10pt}] (8.6,-1.8) -- (8.6,1.8) node [midway,xshift=-0.6cm] {\footnotesize $k$};
\node at (11,-2.5) { $T_k$};

    \node[vertex] (q1) at ( 17,0) {};
    \node[vertex] (q2) at ( 15.3,-1.5) {};
    \node[vertex] (q3) at ( 14.8,-0.5) {};
    \node[vertex] (q4) at ( 14.8, 0.5) {};
    \node[vertex] (q5) at ( 15.3, 1.5) {};
    \node[vertex] (p1) at ( 19,0) {};
    \node[vertex] (p2) at ( 20.7,-1.5) {};
    \node[vertex] (p3) at ( 21.2,-0.5) {};
    \node[vertex] (p4) at ( 21.2,0 .5) {};
    \node[vertex] (p5) at ( 20.7, 1.5) {};

    \begin{scope}[every path/.style={-}]
       \draw (q1) -- (q2);
       \draw (q1) -- (q3);
       \draw (q1) -- (q4);
       \draw (q1) -- (q5);
       \draw[dashed] (q4) -- (q5);
       \draw (p1) -- (p2);
       \draw (p1) -- (p3);
       \draw (p1) -- (p4);
       \draw (p1) -- (p5);
       \draw[dashed] (p4) -- (p5);
       \draw (q1) -- (p1);
     \end{scope} 
\draw [decorate,decoration={brace,amplitude=10pt}] (14.6,-1.8) -- (14.6,1.8) node [midway, xshift = -0.6cm] {\footnotesize $k$};
\draw [decorate,decoration={brace,amplitude=10pt}] (21.4,1.8) -- (21.4,-1.8) node [midway, xshift=0.6cm]  {\footnotesize $l$};
\node at (18,-2.5) { $D_{k,l}$};
  \end{tikzpicture}
}

\captionof{figure}{$T_k$ and $D_{k,l}$}
\label{dt}

\end{figure}
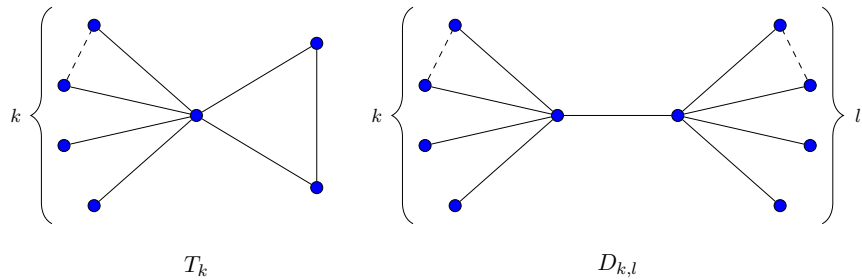

 As in the case of the $P_4$-saturation game, we start by characterising the $P_5$-saturated graphs, which is easily shown by inspection:
\begin{lemma}\label{P5-saturated}
A $P_5$ saturated graph is \emph{either} a vertex-disjoint union of copies of $K_4$, $T_{\geq 0}$, $D_{k, l}$ where $\max(k,l)>0$ and at most one isolated edge \emph{or} a vertex-disjoint union of one isolated vertex and copies of $K_4$.\end{lemma}

\begin{proof}[Proof of Theorem \ref{P5-bound}]
The upper bound is demonstrated by considering the following strategy for Shortener. She will:
\begin{enumerate}[i)]
\item if there are no isolated vertices in $\G_i$ then join two isolated edges to form a $P_4$ if possible, otherwise
\item \label{form5}extend a $P_4 = D_{1,1}$ to a $D_{1,2}$ or extend a $K_{1,3} = D_{0,2}$ to a $D_{1,2}$ or extend a $T_1$ to a $T_2$ if possible, otherwise
\item \label{formK}extend an isolated edge to a $K_{1,2} = D_{0,1}$ if possible, otherwise
\item \label{ext5}extend a component of 5 or more vertices by attaching to it an isolated vertex if possible, otherwise
\item \label{edge}draw an isolated edge if possible, otherwise
\item join two $K_{1,2}$ components into a $D_{2,2}$ if possible, otherwise
\item play arbitrarily.
\end{enumerate}

\begin{claim}\label{P5-structure}
Given this strategy by Shortener, in any graph $\G_t$, there is either at most one component of size four and at most one isolated edge or, there are no components of size four and at most two isolated edges.
\end{claim}

\begin{proof}
We proceed inductively; clearly the condition holds after Prolonger's first move. Suppose it holds after Prolonger moves to $\G_i$. We split into cases according to the existence of isolated vertices.

Suppose first that there are no isolated vertices in $\G_i$. Then the only way to create a 4-vertex component is to join two isolated edges from $\G_i$. But inductively if $\G_i$ already contains a 4-vertex component then there is at most one isolated edge, so a second 4-vertex component cannot be made. Otherwise there are $\leq 2$ isolated edges, so at most one 4-vertex component can be produced. Hence the lemma is satisfied.

Suppose alternatively there is an isolated vertex in $\G_i$. Then we make the stronger claim that \emph{either} there is at most one $P_4$, $K_{1,3}$ or $T_1$ and at most one isolated edge, \emph{or} there are no components of size 4 and at most 2 isolated edges during any time of the game, \emph{or} we are forming the last component of size 4 in the game. Clearly this holds after Prolonger's first move; we proceed inductively.

If there is a $P_4$, $K_{1,3}$ or $T_1$ in $\G_i$ then Shortener will extend it to a $\geq$ 5-vertex component (Step \ref{form5}). Otherwise, if there is an isolated edge in $\G_i$ then Shortener will extend it to a $K_{1,2}$ (Step \ref{formK}). Note from Lemma \ref{P5-saturated} that an isolated vertex can be attached to any $\geq$ 5-vertex component, so otherwise Shortener will attach a vertex to a $\geq 5$-vertex component(Step \ref{ext5}).

Hence $\G_{i+1}$ contains no component of size 4 and at most one isolated edge. So in $\G_{i+2}$ Prolonger can create at most one new 4-vertex component, which must be a copy of $P_4$, $K_{1,3}$ or $T_1$, or at most one isolated edge. Hence our stronger claim remains true for $\G_{i+2}$.
\end{proof}

By Claim~\ref{P5-structure}, there is at most one $K_4$. By Lemma~\ref{P5-saturated} the number of edges does not exceed the number of vertices in any other component. Hence the game score is at most $n + 2$, showing the claimed upper bound.

Before outlining the argument for the lower bound, we define some additional notions. Let us call a component \emph{trivial} if it consists of an isolated vertex. Let us call a non-trivial component \emph{standalone} if it can not be connected to another non-trivial component without completing a $P_5$, otherwise call it non-standalone. Note that for a component to be non-standalone it has to have a vertex which is not the endpoint of a $P_3$. The only $P_5$ free components which have a vertex which is not the endpoint of a $P_3$ are stars (indeed, the second neighbourhood of such a vertex is empty and the first neighbourhood is an independent set). Hence any other component may only be joined to an isolated vertex, as otherwise a $P_5$ will neccessarily appear.

The lower bound is demonstrated by considering the following strategy for Prolonger. He will:

\begin{enumerate}[(i)]
\item {complete a triangle in a $D_{1,2}$ component to make it a $T_2$ or in a $K_{1,3}$ component to make it a $T_1$, or, if not possible}
\item {complete a triangle in a component without a triangle, or, if not possible}
\item {connect two isolated edges to form a $P_4$, or, if not possible}
\item {complete a $K_{1,2}$ component, or, if not possible}
\item {draw an isolated edge, or, if not possible}
\item play arbitrarily.
\end{enumerate}

\begin{claim}\label{standalone}
Given this strategy for Prolonger, the set of star components after his move may be:
empty; or one isolated edge; or one $K_{1,2}$. After Shortener's move, the set of
non-standalone components may be: empty, $K_{1,2}$, $K_{1,3}$, $K_{1,2}$ and an isolated
edge, two isolated edges, or one isolated edge. After Shortener's move, at most a single one of $D_{1,2}$, $K_{1,2}$ or $K_{1,3}$ components exists; if a $D_{1,2}$ exists then at most one isolated edge does.
\end{claim}

\begin{proof}
We induct on the number of moves. The result holds trivially for $G_0$ and $G_{1}$. If the condition holds after Prolonger's move, it can easily be
checked that Shortener's move can only produce sets of stars as stated in the lemma. After Shortener's move, Prolonger will:
\begin{enumerate}[(i)]
\item complete a $K_3$ from a $K_{1,2}$ component, or produce a $T_1$ in the $K_{1,3}$ component, both of which are standalone; 
\item if not possible, he will complete a $P_4$ from two isolated edges, which is standalone; 
\item if not possible, he will complete a $K_{1,2}$ component from one isolated edge; 
\item if not possible, he will draw an isolated edge; 
\item otherwise, there will be at most one isolated edge and no isolated vertices, and he will play arbitrarily but his move will not extend a star into a larger star (otherwise he could have completed a triangle in it), so his edge will be a part of a standalone component. 
\end{enumerate}
In all cases the set of non-standalone components after Prolonger's move is as described in the claim.
\end{proof}

\begin{claim}\label{triangle}
Given this strategy for Prolonger, in $\G_{t^*}$ all standalone components will contain a triangle. The set of non-standalone components will consist of an isolated vertex or an isolated edge.
\end{claim}

\begin{proof}
We claim by induction that after Prolonger's move, there will be at most one component
of size greater than one without a triangle which will be either an isolated edge, a $K_{1,2}$ or a $P_4$. Clearly this holds for $\G_0$, $\G_{1}$ and for this strategy of Prolonger for $\G_2$.
Suppose it holds for after Prolonger's move to $\G_i$. By Claim~\ref{standalone} there is at
most one star component in $\G_i$, so if Shortener connects two components
one of them is an isolated vertex. So in $\G_{i+1}$ the set of non-trivial
components without a triangle will be empty, a $K_2$, a $K_{1,2}$ or $P_4$ or be one
of the preceding and an isolated edge or be a $K_{1,3}$ or a $D_{1,2}$. In each case,
to form $\G_{i+2}$ Prolonger will:

\begin{enumerate}[(i)]
\item complete a triangle in them to create a $T_2$ component or a $T_1$ component or a $K_3$ component or 
\item connect two isolated edges to form a $P_4$ component or 
\item connect an isolated edge to an isolated vertex to form a $K_{1,2}$ component or 
\item create an isolated edge or 
\item else there is at most one non-trivial component without a triangle which can only be an isolated edge and he can play arbitrarily
\end{enumerate}
so the set of non-trivial components without a triangle in $\G_{i+2}$ consists of an isolated edge, a $K_{1,2}$ or a $P_4$. 

Hence Shortener cannot create $D_{k,l}$ components with both $k,l\geq 2$. By Claim~\ref{standalone} there is at most one star component in $\G_i$, so the component would have to be formed via a $D_{1,2}$ or a $K_{1,3}$ component, which are immediately completed into a $T_2$ or $T_1$ component by Prolonger. Hence at the end of the game the non-trivial components without a triangle will be an isolated vertex or an isolated edge, since the other components cannot be a $D_{k,l}$ with $k, l \geq 2$ in $\G_i$ and thus contain a triangle by Claim~\ref{P5-structure}.
\end{proof}

So by Claim~\ref{triangle} all components in $\G_{t^*}$ will contain a triangle except for
at most one isolated edge or isolated vertex. Hence the number of edges in
these components is greater or equal to the number of vertices. Hence
$\G(K_n, P_5) \geq n - 1$.

\end{proof}

%
%
%
%
\section{Game of avoiding all trees on \texorpdfstring{$k$}{k} vertices}

Recall that $\T_k$ is defined to be the family of all trees on $k$ vertices. Consider the game $\G(K_n,\T_k)$. Clearly, the condition that $G$ is $\T_k$-free  is equivalent to requiring that all connected components of $G$ have less than $k$ vertices. Hence being $\T_k$-saturated implies that all components will be cliques of size at most $k - 1$ with any two components having total size at least $k$.

\begin{proof}[Proof of Theorem \ref{Tree-bound}]
Suppose $G$ is $\mathcal{T}_k$-saturated. Then $e(G)$ is a convex quadratic
function of the clique sizes, and so is maximised when all but one clique is
of size $k - 1$. The upper bounds follow immediately.

To demonstrate the lower bounds, suppose that Prolonger chooses two components with the
greatest total number of vertices such that this number is at most $k-1$ and
connects them by an edge.
\begin{claim}
After Prolonger's move, yielding $\G_i$, either (1) there exists at
most one connected component $C_i \subseteq \G_i$ with $1 < |C_i| < k - 1$, or (2) there is an
isolated edge, a connected component of size $k - 2$ and connected
components of size $k - 1$.\end{claim} 

\begin{proof}
The conditions of (1) hold in $\G_0 = E_n$ and $\G_1 = K_2 \sqcup (n - 2)K_1$. We proceed inductively, and split the analysis of Shortener's move into two cases:

\begin{enumerate}[a)]
\item Shortener connects two isolated vertices to make an isolated edge.
\item Shortener does not form an isolated edge, so either no components are
changed in size or $C_i$ is joined to an isolated vertex $u$.
\end{enumerate} 

If Shortener has formed an isolated edge $uv$, then if $|C_i| \leq k - 3$, Prolonger joins it to $uv$ to satisfy the conditions of (1), with $C_{i+2} = C_i \cup \{u,v\}$. If instead there is an isolated vertex $v$ and $|C_i| = k - 2$, then Prolonger joins it to $v$ to satisfy the conditions of (1). Otherwise no component can be extended and the conditions of (2) are satisfied for the rest of the game.

Suppose Shortener has not formed an isolated edge. Then she must have extended $C_i$ or left the component structure unchanged. If there was a set $C_i$, we say $C_{i+1} = C_i$ if she did not extend the component to $u$, and $C_{i+1} = C_i \cup u$ otherwise. If there are no
isolated vertices then no component can be extended and the conditions of (1) are satisfied for the rest of the game. If there is an isolated vertex $v$ and $C_{i+1}$ exists, with $|C_{i+1}| \leq k-2$, Prolonger joins $C_{i+1}$ to $v$ satisfying the conditions of (1). If there is no set $C_i$ or $|C_{i+1}| = k-1$, then if there are two isolated vertices Prolonger joins them to form $C_{i+2}$ satisfying the conditions of (1). If not, then no component can be extended and the conditions of (1) are satisfied for the rest of the game.
\end{proof}

Hence if $n \not\equiv 1 \mod (k - 1)$ the conditions of (2) cannot hold, and since
$G$ is $\mathcal{T}_k$-saturated at the end of the game there cannot be a component of
size $\leq k - 2$ and an isolated vertex. Hence there are $\lfloor\frac{n}{k-1}\rfloor$ $K_{k-1}$'s and one
further clique, which saturates the upper bound.

If $n \equiv 1 \mod (k - 1)$ and $k \geq 3$, then the conditions of (2) could hold, in
which case precisely $k - 2$ edges are lost from removing a vertex from a
$K_{k-1}$ and 1 is gained from an isolated edge. Hence the bound is $k - 3$
below the upper bound.\end{proof}

\section{Forbidding the graph \texorpdfstring{$K_{1,k+1}$}{K\_\{1,k+1\}}}
In lieu of forbidding the family of all trees $T_{k+2}$, we may merely forbid the graph $K_{1, k+1}$. Trivially this corresponds to requiring that in the process, $\Delta(\G_t) \leq k$. From this, we immediately see that in a $K_{1,k+1}$-saturated graph $G$ we have that $\{v \in G : d(v) < k\}$ must form a clique in $G$, as otherwise we could add an edge without producing a $K_{1,k+1}$. Hence, by minimizing the respective quadratic function, we have that the score $\G(K_n, K_{1,k+1}) \geq \frac{1}{2}nk - \frac12 \left(\frac{k+1}{2}\right)^2 $. This lower bound can be improved somewhat.

\begin{proof}[Proof of Theorem \ref{Star-bound}]
The upper bound follows trivially from the fact that $\Delta(G) \leq k$ in any $K_{1,k+1}$-saturated graph $G$. Let Prolonger have the following strategy: Given a graph $\G_i$ by Shortener, she adds the least edge in $\bar{\G_i}$, where the edges $uv$ of $\bar{\G_i}$ are ordered lexicographically by the minimum degree of $u$ and $v$ and then by the maximal degree. Note first that he will attempt to add edges between vertices of degree $\delta(\G_i)$. If Prolonger is unable to find such an edge, then the vertices of degree $\delta(\G_i)$ must form a clique, and hence there are at most $\delta(\G_i) + 1 \leq k+1$ of them. These final $\leq k+1$ vertices may require their degrees to be increased by adding edges to vertices of degree greater than $\delta(\G_i)$.

Consider the graph process given that Prolonger follows this strategy. Let $t_i$ be least such that $\delta(\G_{t_i}) \geq i$ and Shortener has just played. Let $g_i = \sum_v \max(d_{\G_{t_i}}(v) - i, 0)$, if $t_i$ exists, and $g_i = 0$ otherwise. Suppose that $t_i$ exists and that in $\G_{t_i}$ there are $\lambda_i$ vertices of degree $> i$. Then after at most $\left\lceil \frac{1}{2}(n - i - 1 - \lambda_i)\right\rceil$ moves by Prolonger there are $\leq i+1$ vertices of degree $i$, and after at most another $i+1$ moves by Prolonger there are no vertices of degree $i$, unless the game has ended.  Furthermore, note that by parity considerations if the ceiling in phase one increases the bound then the second phase has only $i$ moves.

For the game to end whilst $\delta(\G_t) = i$ requires that there exist $n-i-1$ vertices with degree $k$. Consider the total degree of these vertices between $t_i$ and the hypothesised game end. Note that over Prolonger's $\leq \frac{1}{2}(n - i - 1 - \lambda_i) + i+1$ moves he adds one edge to every vertex of degree $i$, and at most a further $k+1$ to the total degree. We may assume that Shortener only adds edges amongst the $n-i-1$ vertices.

Note that in the first phase of Prolonger's play, his added edges will increase the degree of a subset of the $n-i-1$ vertices, each of degree $i$, by exactly one. So considering just these moves and the initial degree, the sum of degrees over the $n-i-1$ vertices is $(i+1)(n-i-1) + g_i - \lambda_i$, as the first two terms double count the $i+1^{\textrm{th}}$ edge incident on any vertex of degree $> i$ in $\G_{t_i}$. In Shortener's moves, she adds at most $(n-i-1-\lambda_i)$ to this degree sum. In the second phase of Prolonger's play, every pair of moves by Prolonger and Shortener adds at most 3 edges to these vertices. So for the game to last until Prolonger increases $\delta(G_t)$ to $i+1$, it \emph{suffices} to have:
\[
k(n-i-1) \geq  (i+1)(n-i-1) + (g_i - \lambda_i) + (n - i - 1 - \lambda_i) + 3(i+1)
\]
and by the same degree counting we have:
\[g_{i+1} \leq (g_i - \lambda_i) + (n - i - 1 - \lambda_i) + 3(i+1).\]
Note that $\lambda_i \geq g_i/(k - i)$, as any vertex contributes at most $k-i$ to $g_i$ and $\lambda_i$ counts the number of non-zero contributions to $g_i$. Define:
\[f_0 = 0,\quad f_{i+1} = f_i + (n + 2k + 2) - 2f_i / (k - i).\]
For all $i \leq k-2$, $f_{i+1}$ is increasing in $f_i$. Note that $g_0 = 0$, and hence:
\[g_{i+1} \leq g_i + (n + 2i + 2) - 2\lambda_i \leq g_i + (n + 2k + 2) - \frac{2g_i}{k - i} \leq f_{i+1}\]
for all $i \leq k-2$, with the last inequality following by induction on $i$. Note also that $t_{i+1}$ exists if $(k-i-1)(n-i-1) \geq f_{i+1}$, as we have $f_{i+1} \geq (g_i - \lambda_i) + (n - i - 1 - \lambda_i) + 3(i+1)$ from this inequality. Additionally, we have $f_i = i(n + 2k + 2)\frac{k-i}{k-1}$ by induction. Hence to show that $t_i$ exists for all $i \leq k-2$ it suffices that:
\[i(n + 2k + 2)\frac{k-i}{k-1} \leq (k-i)(n-i) \Leftrightarrow i \leq \frac{n(k-1)}{n+3k+1}\]
holds for all $i \leq k-2$. Hence for 

So for $n \geq (3k+1)(k-2)$, we have that $t_i$ exists for all $i \leq k-2$, and so the minimum degree of the saturated graph is at least $k-2$. Hence $\G(K_n, K_{1,k+1}) \geq \frac{1}{2}\left(kn-2(k-1)\right)$ as required.
\end{proof}

\section{Concluding Remarks}
There remain many interesting open problems, chiefly the resolution of the triangle saturation game. Given a graph $G$ which is not a tree, providing effective bounds on the $\mathcal{G}(K_n, G)$ would be highly desirable. In our results, we show that a careful analysis of the maximal components is of substantive use, and that the supply of low degree vertices controls the ability of both players to enforce conditions on the game. However, our results are strongly predicated on finding explicit strategies; for example, the component structure of $P_6$-saturated graphs is not hard to determine but finding an explicit strategy in this case seems hard.

\end{document}